\documentclass{amsart}
\usepackage[utf8]{inputenc}
\usepackage[T1]{fontenc}
\usepackage{amsmath, amsthm,amssymb,amscd,verbatim}
\usepackage[mathscr]{eucal}
\usepackage{enumerate,multicol}
\usepackage{hyperref}

\usepackage{color}

\newtheorem{theorem}{Theorem}[section]
\newtheorem{corollary}[theorem]{Corollary}
\newtheorem{proposition}[theorem]{Proposition}

\theoremstyle{definition}
\newtheorem{definition}[theorem]{Definition}
\newtheorem{remark}[theorem]{Remark}
\newtheorem{question}[theorem]{Question}
\newtheorem{lemma}[theorem]{Lemma}
\newtheorem{fact}[theorem]{Fact}

\newtheorem{warning}[theorem]{Warning}

\def\conv{\operatorname{conv}}
\def\tp{\operatorname{tp}}
\def\Av{\operatorname{Av}}

\def\lc{\left\lceil}   
\def\rc{\right\rceil}
\let\strokeL\L
\renewcommand\L{\mathbf{L}}
\newcommand{\floor}[1]{\left\lfloor #1 \right\rfloor}

\title{Concerning Keisler Measures over ultraproducts}

\author[K. Gannon]{Kyle Gannon}

\address{Beijing International Center for Mathematical Research (BICMR)\\
Peking University\\
Beijing, China.}
\email{kgannon@bicmr.pku.edu.cn}



\begin{document}

\begin{abstract} As consequence of the VC theorem, any pseudo-finite measure over an NIP ultraproduct is generically stable. We demonstrate a converse of this theorem and prove that any finitely approximable measure over an ultraproduct is itself pseudo-finite (even without the NIP assumption). We also analyze the connection between the Morley product and the pseudo-finite product. In particular, we show that if $\mu$ is definable and both $\mu$ and $\nu$ are pseudo-finite, then the Morley product of $\mu$ and $\nu$ agrees with the pseudo-finite product of $\mu$ and $\nu$. Using this observation, we construct generically stable idempotent measures on pseudo-finite NIP groups. 

\end{abstract}

\maketitle

\section{Introduction}

First, we remind the reader that Keisler measures over ultraproducts have been studied by many people in many contexts (e.g. \cite{Zoe,Zoe2,Chern2,GA,Garcia,Isaac,Hrushovski,Mac,Mal,Si,Star}).

In the NIP setting, there are several methods for constructing generically stable measures. One such method is via ultralimits. The following theorem is a consequence of the VC theorem (e.g. \cite[Example 7.32]{Sibook} or \cite[Corollary 1.3]{Si} for similar statements):

\begin{theorem}\label{fact:ultra} Let $(M_i)_{i \in I}$ be an indexed family of $\mathcal{L}$-structures, $D$ be an ultrafilter on $I$, and $\mathcal{M} := \prod_{D} M_i$ the be the ultraproduct. Suppose that $\mathcal{M}$ is NIP.  
For each $i \in I$, let $\mu_i$ be a Keisler measure on $M_i$ which concentrates on finitely many realized types. If $\mu := \lim_{D} \mu_i$ is the ultralimit (a Keisler measure on $\mathcal{M}$) then $\mu$ is finitely approximated, i.e. 
for every formula $\varphi(x,y) \in \mathcal{L}$ and $\epsilon > 0$,
 there exists $\overline{\mathbf{a}} := \mathbf{a}_1,...,\mathbf{a}_n$ in $\mathcal{M}^{x}$ such that 
\begin{equation*} 
\sup_{\mathbf{b} \in \mathcal{M}^{y}} |\mu(\varphi(x,\mathbf{b})) - \Av(\overline{\mathbf{a}})(\varphi(x,\mathbf{b}))| < \epsilon. 
\end{equation*} 
In other words, suppose that $\mathcal{M}$ is NIP and $\mu \in \mathfrak{M}_{x}(\mathcal{M})$. If $\mu$ is pseudo-finite then $\mu$ is finitely approximated. 
\end{theorem} 

The first goal of this paper is to prove a converse to Theorem \ref{fact:ultra} (answering a question of Chernikov and Starchenko). We show that if $\mathcal{L}$ is countable and our ultrafilter $D$ is countably incomplete then any finitely approximated measure over our ultraproduct is actually pseudo-finite (Theorem \ref{Theorem:pseudo}). We remark that the second assumption, i.e. the ultrafilter being countably incomplete, is quite tame since one needs to assume the existence of inaccessible cardinals to have a countably complete ultrafilter.  We also remark that this result does not require an NIP assumption.

Secondly, when working over an ultraproduct $\mathcal{M}$, there are two notions of product for measures. There is the Morley product (denoted $\otimes$) and the pseudo-finite product (denoted $\boxtimes$). We remark that one runs into an immediate problem when trying to compare these two products since these operations are often defined over different collections of measures. The pseudo-finite product is defined on pairs of pseudo-finite measures while the Morley product is usually defined on pairs of global Keisler measures which are Borel-definable over a small submodel. However, if one adds a definability assumption, then one can recover a version of the Morley product \textit{over arbitrary models}. To be explicit, we prove that if $\mu$ and $\nu$ are pseudo-finite measures over an ultraproduct and $\mu$ is definable, then the Morley product and pseudo-finite product agree (Theorem \ref{MT2}). As a consequence, one can show that if $\nu$ is also definable, then $\mu \otimes \nu = \nu \otimes \mu$ (Corollary \ref{fun}). We then apply local versions of these results to demonstrate that certain measures are not pseudo-finite over ultraproducts of the Paley graphs (Theorem \ref{Paley}). Finally, we return to the classical Keisler measure setting (measures over a monster model) and use Theorems \ref{fact:ultra} and \ref{MT2} to build generically stable idempotent measures in the context of pseudo-finite NIP groups (Theorem \ref{gsim}).

We begin with a preliminaries section. The rest of the paper follows the outline presented in the preceding paragraphs.

\subsection*{Acknowledgements} We thank Artem Chernikov for both specific and general comments as well as discussion.

\section{Preliminaries} 
If $r$ and $s$ are real numbers and $\epsilon$ is a real number greater than $0$, then we write $r \approx_{\epsilon} s$ to mean $|r - s| < \epsilon$ and $r \approxeq_{\epsilon} s$ to mean $|r - s| \leq \epsilon$.

For the most part, our notation is standard. The symbol $\mathcal{L}$ denotes a first order language. The letters $x,y,z$ denote finite tuples of variables. We use the letter $M$ to denote an arbitrary $\mathcal{L}$-structure. If $A \subseteq M$, then we let $\mathcal{L}(A)$ be the collection of formulas with parameters from $A$ (modulo semantic equivalence, i.e. two formulas $\varphi(x)$ and $\psi(x)$ are identified if they define the same definable subsets of $M^{x}$). A formula in $\mathcal{L}(A)$ is called an ``$\mathcal{L}(A)$-formula''. An $\mathcal{L}$-formula is a $\mathcal{L}(\emptyset)$-formula. We use $\mathcal{L}_{x}(A)$ to denote the $\mathcal{L}(A)$-formulas with free variable(s) $x$. We write partitioned formulas as $\varphi(x,y)$ with \textit{variable} variables $x$ and \textit{parameter} variables $y$. We let $\varphi^{*}(y,x)$ denote the exact same formula as $\varphi(x,y)$ but with the \emph{variable} variables and \emph{parameter} variables swapped.  

Unlike similar papers about Keisler measures, we do not identify a type and its corresponding Keisler measure. Let $A$ be a subset of a model $M$. We let $S_{x}(A)$ denote the space of types over $A$ (in variable(s) $x$). We let $\mathfrak{M}_{x}(A)$ denote the space of Keisler measures over $A$ (in variable(s) $x$). For any (tuple of) variable(s) $x$, and any subset $A \subseteq M$, we have a map $\delta: S_{x}(A) \to \mathfrak{M}_{x}(M)$ via $\delta(p) = \delta_{p}$ where $\delta_{p}$ is the \textit{Dirac measure at the type $p$}. We sometimes refer to $\delta_{p}$ as the \textit{corresponding Keisler measure} of $p$. If $\overline{a} = a_1,...,a_n$ is a sequence of points in $M^{x}$, then we let $\Av(\overline{a})$ be the associated average measure in $\mathfrak{M}_{x}(M)$. Explicitly, for any $\psi(x) \in \mathcal{L}_{x}(M)$, we define
 \begin{equation*}\Av(\overline{a})(\psi(x)) = \frac{|\{1\leq i \leq n: \mathcal{U} \models \psi(a_i)\}|}{n}.
 \end{equation*} 
 Moreover we let
 \begin{equation*} 
 \conv_{x}(M) := \Big\{ \sum_{i=1}^n r_i \delta_{\tp(a_i/M)}: a_i \in M^{x}, n \in \mathbb{N}_{\geq 1}, \sum_{i=1}^n r_i = 1, r_i > 0\Big\} \subseteq \mathfrak{M}_{x}(M). 
 \end{equation*}
 
\subsection{Ultraproducts, ultralimits and pseudo-finite measures} We first discuss our conventions on ultraproducts. Let $I$ be an (infinite) indexing set, $(M_i)_{i \in I}$ be an indexed family of $\mathcal{L}$-structures, and $D$ be an ultrafilter on $I$. $D$ will always denote a non-principal ultrafilter. We denote the ultraproduct of $(M_i)_{i \in I}$ relative to $D$ as $\prod_{D} M_i$. In practice, we will usually write ultraproducts as $\mathcal{M}$ when the indexing set, sequence of models, and ultrafilter are unambiguous. We will write elements of $\mathcal{M}$ as $\mathbf{a},\mathbf{b},\mathbf{c}$, etc. If $(c_i)_{i \in I}$ is an indexed family of elements such that $c_i \in M_i$, then we let $[(c_i)]_{D}$ denote the corresponding element in $\mathcal{M}$, i.e. the equivalence class of $(c_i)_{i \in I}$ modulo $D$. If $\mathbf{b} \in \mathcal{M}$, then we write $(\mathbf{b}(i))_{i \in I}$ or simply $\mathbf{b}(i)$ to denote a (choice of a) sequence such that each $\mathbf{b}(i) \in M_i$ and $[\mathbf{b}(i)]_{D} = \mathbf{b}$. 

We recall the definition of an ultralimit. Since this paper is primarily concerned with finitely additive measures, we will restrict our definition to the case of ultralimits of indexed families of real numbers.

\begin{definition} Let $(r_i)_{i \in I}$ be a bounded indexed family of real numbers and let $D$ be an ultrafilter on $I$. Then the ultralimit of $(r_i)_{i \in I}$ (with respect to $D$) is the unique real number $r$ such that for every $\epsilon > 0$, $\{i \in I: |r_i - r| \leq \epsilon \} \in D$. We denote this ultralimit as $\lim_{D} r_i$. 
\end{definition} 

We remind the reader that ultralimits of bounded indexed families of real numbers always exist. The following facts are elementary. 

\begin{fact}\label{computation} Let $r_1,...,r_n$ be real numbers, let $(s_i^{1})_{i \in I},...,(s_{i}^{n})_{i \in I}$ be sequence of real numbers, and $D$ be an ultrafilter on $I$. Then the following hold:
\begin{enumerate} 
\item $|\lim_{D} s^{1}_i| = \lim_{D} |s^{1}_i|$. 
\item $\sum_{j=1}^{n} r_i \lim_{D} s_i^{j} = \lim_{D} \sum_{j=1}^{n} r_i \cdot s_i^{j}.$
\item If there exists $A \in D$ such that for any $i \in A$, $s_i^{1} \geq r_1$, then $\lim_{D} s_i^{1} \geq r_1$. 
\item If there exists $A \in D$ such that for any $i \in A$, $s_i^{1} \leq r_1$, then $\lim_{D} s_i^{1} \leq r_1$. 
\item If there exists $A \in D$ such that for any $i \in A$, $|s_i^{1} - s_i^{2}| \leq \epsilon$,
 then 
 \begin{equation*} \lim_{D} s_i^{1} \approxeq_{\epsilon} \lim_{D} s_i^{2}. 
 \end{equation*}
\end{enumerate} 
\end{fact} 

We now recall the definition of an ultralimit of a family of measures as well as the definition of a pseudo-finite measure/type. We remark that the ultralimit can be defined on the sigma-algebra generated by internal subsets of the ultraproduct (Loeb measure construction). However, since we only care about Keisler measures, we restrict our discussion to definable subsets of the ultraproduct. 

\begin{definition} Let $(M_i)_{i \in I}$ be a family of $\mathcal{L}$-structures and $D$ be an ultrafilter on $I$. Let $\mathcal{M} := \prod_{D} M_i$ and for each $i$, we let $\mu_i \in \mathfrak{M}_{x}(M_i)$. The \textbf{ultralimit} of $(\mu_i)_{i \in I}$ (relative to $D$) is the unique measure $\mu \in \mathfrak{M}_{x}(\mathcal{M})$ such that for every formula $\varphi(x, \mathbf{b}) \in \mathcal{L}_{x}(\mathcal{M})$ we have that 
\begin{equation*} \mu(\varphi(x,\mathbf{b})) = \lim_{D} \mu_{i}(\varphi(x,\mathbf{b}(i))). 
\end{equation*} 
We usually denote the measure $\mu$ defined above as $\lim_{D} \mu_i$. Moreover we say that a measure $\mu \in \mathfrak{M}_{x}(\mathcal{M})$ is \textbf{pseudo-finite} (with respect to the indexed family $(M_i)_{i \in I}$ and ultrafilter $D$) if there exists an indexed family of measures $(\mu_i)_{i \in I}$ such that $\mu_i \in \conv_{x}(M_i)$ and $\lim_{D} \mu_i = \mu$. Finally, we say that the type $p$ is \textbf{pseudo-finite} if the corresponding Keisler measure $\delta_{p}$ is pseudo-finite. 
\end{definition} 

The following fact is straightforward and left to the reader as an exercise. 

\begin{fact}\label{prop:well}The ultralimit $\lim_{D} \mu_i$ in the above definition is well-defined. 
\end{fact}

Finally we recall the definition of a countably incomplete ultrafilter. 

\begin{definition} Let $I$ be an indexing set and $D$ an ultrafilter on $I$. We say that $D$ is countably incomplete if there exists a subset $\{X_i: i \in \mathbb{N}\}$ of $D$ such that $\bigcap_{i \in \mathbb{N}} X_i = \emptyset$. 
\end{definition}

\subsection{Working over models.} A lot of literature dealing with Keisler measures focuses on global measures, i.e. those which are defined over a monster model. On the other hand this paper is concerned with ultraproducts which are generally not monster models. Therefore we need to set some conventions. 

\begin{definition} Let $M$ be any model of $T$, $p \in S_{x}(M)$ and $\mu \in \mathfrak{M}_{x}(M)$. 
\begin{enumerate}
\item We say that $p$ is $M$-\textbf{definable} if for every $\mathcal{L}$-formula $\varphi(x,y)$, there exists a formula $d^{\theta}_{p}(y) \in \mathcal{L}_{y}(M)$ such that for any $b \in M$, 
\begin{equation*} 
\theta(x,b) \in p \iff M \models d^{\theta}_{p}(b). 
\end{equation*} 
When $M$ is obvious, we sometimes simply write that $p$ is definable. 
\item We say that $\mu$ is $M$-\textbf{definable} if for every $\mathcal{L}$-formula $\varphi(x,y)$, there exists a continuous function $F^{\varphi}_{\mu,M}:S_{y}(M) \to [0,1]$ such that for every $b \in M^{y}$, $F^{\varphi}_{\mu,M}(\tp(b/M^{y})) = \mu(\varphi(x,b))$. We remark that if such a continuous function exists, then it is the unique continuous function with this property. When $M$ is obvious, we sometimes simply write that $\mu$ is definable. 
\item We say that $\mu$ is \textbf{finitely approximated} if for every partitioned $\mathcal{L}$-formula $\varphi(x,y)$ and every $\epsilon > 0$, there exists a sequence $a_{1},...,{a}_n \in M^{x}$ such that 
\begin{equation*}
    \sup_{b \in M^{y}}|\mu(\varphi(x,b)) - \Av(\overline{a})(\varphi(x,b))| < \epsilon.
\end{equation*}
\item We say that $p$ is \textbf{finitely approximated} if the corresponding Keisler measure $\delta_{p}$ is finitely approximated. 
\end{enumerate} 
\end{definition} 

One might (and should) wonder why we are restricting ourselves to $M$-\textit{definable} measures and not working with ``$M$-\textit{Borel-definable}" measures. This is because being Borel-definable does not have such a good ``over-a-model" analogue. We discuss why in the warning below.

\begin{warning} One may be tempted to define \textit{Borel-definable over a model} in a similar way to definability. A plausible definition might be the following: We say that a measure $\mu \in \mathfrak{M}_{x}(M)$ is \textit{M-Borel-definable$^*$} if for every formula $\varphi(x,y) \in \mathcal{L}(M)$, there exists a unique Borel function $f_{\mu}^{\varphi}: S_{y}(M) \to [0,1]$ such that for any $b \in M^{y}$, $f_{\mu}^{\varphi}(\tp(b/M)) = \mu(\varphi(x,b))$. The reason why we require uniqueness is so that we can construct a unique \textit{Morley product} with any measure in $\mathfrak{M}_{y}(M)$ (see Definition \ref{def:morley}). 

However, a measure $\mu$ can have \emph{multiple Borel-definable extensions}. In particular, there exists a model $M$ and a measure $\mu \in \mathfrak{M}_{x}(M)$ such that $\mu$ is $M$-definable but not $M$-Borel-definable$^*$. Consider $M := (\mathbb{R},<)$ and $p$ is the unique complete type extending $\{x > a: a \in \mathbb{R}\}$. The measure $\mu = \delta_{p}$ is $\mathbb{R}$-definable but not $\mathbb{R}$-Borel-definable$^*$ essentially because both the global heir and global coheir are Borel-definable extensions of $p$ and so uniqueness fails. 
\end{warning}

The following facts are straightforward and left to the reader as an exercise. 

\begin{fact}\label{big:fact} Let $M$ be any model of $T$, $p \in S_{x}(M)$, and $\mu \in \mathfrak{M}_{x}(M)$.
\begin{enumerate} 
\item The type $p$ is definable if and only if the measure $\delta_{p}$ is definable. 
\item If $\mu$ is finitely approximated, then $\mu$ is definable.
\item If $p$ is finitely approximated, then $p$ is definable. 
\end{enumerate} 
\end{fact} 

\begin{proposition}\label{fam:uniform} Suppose that $\mu$ is finitely approximated. Then for any finite collection of $\mathcal{L}$-formulas $\theta_1(x,y_1),..., \theta_n(x,y_{n})$, there exists $\overline{a} := a_1,...,a_n \in M^{x}$ such that 
\begin{equation*} \sup_{c \in M^{y_i}} |\mu(\theta_{i}(x,c)) - \Av(\overline{a})(\theta_i(x,c))| < \frac{1}{n}. 
\end{equation*} 
\end{proposition} 

\begin{proof} Follows from a standard encoding argument. Consider the new variables $z_*,z_1,...,z_n$ where $|z_*| = |z_i| = 1$. Now let $\bar{y} = (y_0,y_1,...,z_*,z_1,...,z_n)$ and define the formula
\begin{equation*}
    \gamma(x,\bar{y}) : =  \bigwedge_{i \leq n}\left( \left( z_* = z_i \wedge \bigwedge_{j \leq n ; j \neq i}^{n} z_j \neq z_*  \right)  \to \theta_i(x,y_i) \right). 
\end{equation*} By finite approximability, there exists $\overline{a} \in (M^{x})^{<\omega}$ such that for any $\bar{b} \in \mathcal{U}^{\bar{y}}$, 
\begin{equation*} |\mu(\gamma(x,\bar{b})) - \Av(\overline{a})(\gamma(x,\bar{b}))| < \frac{1}{n}. 
\end{equation*}
We claim $\Av(\overline{a})$ satisfies the condition. 
\end{proof}

\section{Finitely Approximated implies pseudo-finite}

The purpose of this section is to prove a (partial) converse to Theorem \ref{fact:ultra}. We will show that if a measure $\mu$ on an ultraproduct $\mathcal{M}$ is finitely approximated, then $\mu$ is pseudo-finite (modulo some minimal assumptions). Our result is partial since we require our language to be countable and our ultrafilter to be countably incomplete. However this result does not require an NIP assumption. For clarity, we first provide a version of the proof for types over a countable index set before proving the general case for measures. The measure case is structurally the same as the type case, however we feel that the type case is pedagogically helpful.

\begin{theorem}\label{types} Let $\mathcal{L}$ be a countable language, $I = \mathbb{N}$, $D$ be a non-principal ultrafilter on $I$, and $(M_i)_{i \in I}$ be an indexed family of $\mathcal{L}$-structures.  Let $\mathcal{M} = \prod_{D}M_i$ and $p \in S_{x}(\mathcal{M})$. If $\delta_p$ is finitely approximated then $\delta_{p}$ is pseudo-finite. 
\end{theorem} 

\begin{proof} Let $(\varphi_i(x,y_i))_{i \in \mathbb{N}}$ be an enumeration of partitioned $\mathcal{L}$-formulas. Since $\delta_{p}$ is finitely approximated,  $\delta_{p}$ is definable (by Proposition \ref{big:fact}). Hence for each partitioned $\mathcal{L}$-formula $\varphi_{i}(x,y_{i})$, there exists an $\mathcal{L}(\mathcal{M})$-formula $d_{p}^{\varphi_{i}}(y_{i},\mathbf{c}_{i})$ such that for any $\mathbf{b} \in \mathcal{M}^{|y_{i}|}$, $\mathcal{M} \models d_{p}^{\varphi_{i}}(\mathbf{b},\mathbf{c}_{i})$ if and only if $\varphi_{i}(x,\mathbf{b}) \in p$. Since $\delta_p$ is finitely approximated, it follows from Proposition \ref{fam:uniform} that for every $n \in \mathbb{N}$, there exists some $\overline{\mathbf{a}} \in (\mathcal{M})^{< \omega}$ such that for every $i \leq n$, 
 
\begin{equation*} \sup_{\mathbf{b} \in \mathcal{M}^{|y_{i}|}}|\delta_{p}(\varphi_{i}(x,\mathbf{b})) - \Av(\overline{\mathbf{a}})(\varphi_{i}(x,\mathbf{b}))| < \frac{1}{n}. 
\end{equation*}
Therefore for each $n \in \mathbb{N}$ there exists some number $r(n)$ such that
\begin{align*}  \mathcal{M} \models \exists x_1,...,x_{r(n)} \\  \bigwedge_{i \leq n} \Big(\forall y_{i} \Big( & d_{p}^{\varphi_{i}}(y_{i},\mathbf{c}_{i})  \leftrightarrow \bigvee_{A \in \Gamma(n)}  \Big( \bigwedge_{s \in A} \varphi_{i}(x_s,y_{i}) \wedge \bigwedge_{t \not \in A}  \neg \varphi_{i}(x_t,y_i)\Big)\Big)\Big),
\end{align*} 
where $\Gamma(n) = \{A \subset \{1,...,r(n)\}: \frac{|A|}{r(n)} > 1 - \frac{1}{n}\}$.  For each $n$, we let $\psi_{n} = \exists \overline{x} \sigma_{n}(\overline{x})$ be the sentence described above.
This sentence states that there exists $r(n)$ elements in $\mathcal{M}$ whose average uniformly approximates the measure $\delta_{p}$ on the formulas $\varphi_{1}(x,y_1),...,\varphi_{n}(x,y_n)$ up to an error of $\frac{1}{n}$. 
For each $i \in \mathbb{N}$ choose a sequence $(c_{i_j})_{j \in \mathbb{N}}$ such that  $[(c_{i_j})]_{D} = \mathbf{c}_{i}$. By \strokeL o\'s's theorem, it follow that
\begin{align*} E_n = \Big\{ k \in &  \mathbb{N} :   M_{k} \models \exists x_1,...,x_{r(n)}  \\  \bigwedge_{i \leq n}  &  \Big(\forall y_{i} \Big( d_{p}^{\varphi_{i}}(y_{i},c_{i_k}) \leftrightarrow \bigvee_{A \in \Gamma(n)} \Big( \bigwedge_{s \in A}  \varphi_{i}(x_s,y_{i}) \wedge \bigwedge_{t \not \in A} \neg \varphi_{i}(x_t,y_i)\Big)\Big)\Big)\Big\} \in D.
\end{align*} 
Let $X_n = \{i \in \mathbb{N}: i \geq n\}$ and let $Y_n = X_n \cap E_n$. Notice that $\bigcap Y_n = \emptyset$ and $Y_n \in D$ for each $n \in \mathbb{N}$. We now construct a sequence of Keisler measures $(\mu_{j})_{j \in \mathbb{N}}$ such that each $\mu_j$ is in $\conv_{x}(M_j)$. 
We then argue that $\lim_{D} \mu_{j} = \delta_{p}$. Consider the function $g: \mathbb{N} \to \mathbb{N}$ where 
\begin{equation*}
g(j)=\begin{cases}
\begin{array}{cc}
\max\{n\in\mathbb{N}:j\in Y_{n}\} & \text{if exists,}\\
0 & \text{otherwise.}
\end{array}\end{cases}
\end{equation*} 
If $g(j) = 0$, set $\mu_j$ equal to any measure in $\conv_{x}(M_j)$. If $g(j) = m > 0$, then 
\begin{align*} M_{j} \models \exists x_1,...,x_{r(m)} \\  \bigwedge_{i \leq m}   \Big(\forall y_{i}  & \Big( d_{p}^{\varphi_{i}}(y_{i},c_{i_j}) \leftrightarrow \bigvee_{A \in \Gamma(m)} \Big( \bigwedge_{s \in A} \varphi_{i}(x_s,y_{i}) \wedge \bigwedge_{t \not \in A} \neg \varphi_{i}(x_{t},y_i) \Big)\Big). 
\end{align*} 
Choose $\overline{a}_{j} = (a_{1}^{j},...,a_{r(n)}^{j})$ in $M_j^{r(n)}$ which satisfy the formula above (after removing the existential quantifiers $\exists x_1,...,x_{r(m)}$). Let $\mu_{j} = \Av(\overline{a}_j)$. 

We now demonstrate that the ultralimit of $(\mu_{j})_{j \in \mathbb{N}}$ is equal to $\delta_{p}$. Fix $\theta(x,\mathbf{e}) \in \mathcal{L}_{x}(\mathcal{M})$. Without loss of generality, assume that $\delta_{p}(\theta(x,\mathbf{e})) = 1$. Fix $\epsilon > 0$ and let $(e_i)_{i \in \mathbb{N}}$ be a sequence such that $[(e_i)]_{D} = \mathbf{e}$. By Fact \ref{computation}, it suffices to show that $K = \{j \in \mathbb{N}: \mu_j(\theta(x,e_j)) \geq 1 - \epsilon \} \in D$. Notice that $\theta(x,y)$ appears in our index, say $\theta(x,y) = \varphi_{m}(x,y_{m})$. Since $\delta_{p}(\theta(x,\mathbf{e})) = 1$ we have that $\mathcal{M} \models d_{p}^{\varphi_{m}}(\mathbf{e},\mathbf{c}_m)$ and so the set $Z=\{j\in\mathbb{N}:M_{j}\models d_{p}^{\varphi_{m}}(e_{j},c_{m_j})\}$ is in $D$. 
Let $m_{*} = \max\{m, \lc \frac{1}{\epsilon} +1 \rc\}$. By construction $Y_{m_*} \cap Z \in D$. We argue that $Y_{m_*} \cap Z \subseteq K$ and so $K$ is in $D$ (which will complete our proof).

 Let $j \in Y_{m_*} \cap Z$. Since $j \in Y_{m_{*}}$, it follows that 
 
 \begin{equation*} M_j \models d_{p}^{\varphi_m}(e_j,c_{m_j}) \leftrightarrow  \bigvee_{A \in \Gamma(m_*)}\Big(\bigwedge_{s \in A}  \varphi_{m}(a^{j}_{s},e_j) \wedge \bigwedge_{t \not \in A} \neg \varphi_{m}(a_{t}^{j},e_j)\Big).
 \end{equation*} 
Since $j \in Z$, it follows that $M_{j} \models d_{p}^{\varphi_m}(e_j, c_{m_j})$. So, $M_j \models \bigwedge_{s \in A_{*}}  \varphi_{m}(a^{j}_{s},e_j)$ for some $A_{*} \in \Gamma(m_*)$. Hence 
\begin{equation*} \mu_{j}(\theta(x,e_j)) = \frac{|\{s \leq r(m_{*}) :M_j \models \varphi_{m}(a_{s}^{j},e_j)\}|}{r(m_{*})}  \geq \frac{|A_*|}{r(m_*)} > 1 - \frac{1}{m_{*}} > 1 - \epsilon. 
\end{equation*} 
Thus $Y_{m_*} \cap Z \subset K$ and so $K \in D$. 
\end{proof} 

\begin{remark} In the proof above, we used $I = \mathbb{N}$. However, we only need our ultrafilter to be countably incomplete for the result to go through. We will see this in the general version of the proof.
\end{remark} 

The following fact is stated without proof for global definable measures in \cite[Fact 2.3]{ChernGan}. We provide a short proof for completeness. 

\begin{proposition}\label{prop:def} Let $\mu \in \mathfrak{M}_{x}(M)$ and suppose that $\mu$ is $M$-definable. Then for every natural number $n$ (say greater than 4) and formula $\varphi(x,y) \in \mathcal{L}$, there exists definable sets $d^{\varphi}_{0,n}(y),...,d^{\varphi}_{n,n}(y)$ in $\mathcal{L}_{y}(M)$ such that 
\begin{enumerate}
\item The collection $\{d^{\varphi}_{i,n}(y)\}_{i=0}^{n}$ cover $M^{y}$, i.e. $M \models \forall y \left( \bigvee_{i=0}^{n} d_{i,n}^{\varphi}(y) \right)$. 
\item If $M \models d^{\varphi}_{i,n}(b)$, then $|\mu(\varphi(x,b)) - \frac{i}{n}| < \frac{1}{n}$. 
\item If $|\mu(\varphi(x,b)) - \frac{k}{n}| < \frac{1}{n}$, then $M \models \bigvee_{j=k-1}^{k+1} d_{j,n}^{\varphi}(b)$. 
\end{enumerate} 
\end{proposition}

\begin{proof} Since $\mu$ is definable over $M$, the map $F_{\mu,M}^{\varphi}: S_{y}(M) \to [0,1]$ is continuous. Consider the sets $I_{n} = \{0,\frac{1}{n},\frac{2}{n},...,\frac{n-1}{n},n\}$. Let $A_0 = \left(F_{\mu}^{\varphi} \right)^{-1}([0,\frac{1}{n}))$, $A_n = \left(F_{\mu}^{\varphi} \right)^{-1}((\frac{n-1}{n},1])$, and for $0 < i < n$, let $A_{i} = \left(F_{\mu}^{\varphi} \right)^{-1}((\frac{i-1}{n},\frac{i+1}{n}))$. Since $F_{\mu}^{\varphi}$ is continuous, we have that for each $0 \leq i \leq n$, $A_i = \bigvee_{j \in J_i}[\theta_{j}^{i}(y)]$. The collection $\{[\theta_{j}^{i}(y)]: j \in J_{i}; i \in I_{n}\}$ forms an open cover of $S_{y}(M)$. So we have a finite subcover say $\{[\theta_{j}^{i}(y)]: j \in  J_{i}^{*}; i \in I_{n}\}$ where $J_{i}^{*}$ is a finite subset of $J_{i}$. For each $0 \leq i \leq n$, we define
\begin{equation*} 
d_{i,n}^{\varphi}(y) =\begin{cases}
\begin{array}{cc}
\bigvee_{j \in J_{i}^{*}} \theta_{j}^{0}(y) & J_{i}^{*} \neq \emptyset,\\
y \neq y  & \text{otherwise}.
\end{array}\end{cases}
\end{equation*} 
We claim that $\{d_{i,n}^{\varphi}(y)\}_{i=0}^{n}$ as constructed satisfy the theorem. 
\end{proof}  

We also need the following fact.

\begin{fact}\label{fact:elementary} Suppose $m \in \mathbb{N}_{>0}$ and $r,q \in \mathbb{R}$. If $|q - r/m| < 1/m$ then 
\begin{equation*} 
 \left|q - \frac{\floor{r}}{m} \right| < \frac{1}{m} \text{ or } \left|q - \frac{\lc r \rc}{m} \right| < \frac{1}{m}. 
\end{equation*} 
\end{fact}

\begin{theorem}\label{Theorem:pseudo} Let $\mathcal{L}$ be a countable language, $I = \kappa$, $D$ be a countably incomplete ultrafilter on $I$, and $(M_i)_{i \in I}$ be an indexed family of $\mathcal{L}$-structures. Let $\mathcal{M} = \prod_{D}M_i$ and $\mu \in \mathfrak{M}_{x}(\mathcal{M})$. If $\mu$ is finitely approximated then $\mu$ is pseudo-finite. 
\end{theorem} 

\begin{proof} The proof of this theorem is structurally similar to the proof of Proposition \ref{types}, but slightly more complicated since it is more difficult to encode the measure of a formula into the language. Let $(\varphi_{i}(x,y_i))_{i \in \mathbb{N}}$ be an enumeration of partitioned $\mathcal{L}$-formulas. Since $\mu$ is finitely approximated, $\mu$ is definable (Proposition \ref{big:fact}). By Propositions \ref{prop:def}, for each partitioned $\mathcal{L}$-formula $\varphi_{i}(x,y_{i})$ and $n \in \mathbb{N}$, there exists $\mathcal{L}_{y_i}(\mathcal{M})$-formulas (not necessarily a partition) $d_{0,n}^{\varphi_{i}}(y_{i},\mathbf{c}_{0,n}^{i}),...,d_{n,n}^{\varphi_{i}}(y_{i},\mathbf{c}_{n,n}^{i})$ such that 
\begin{enumerate}
\item  $\{d_{l,n}^{\varphi_i}(y_{i},\mathbf{c}_{l,n}^{i})\}_{l=0}^{n}$ covers $\mathcal{M}^{|y_i|}$,
\item  if $\mathcal{M}\models d_{j,n}^{\varphi_{i}}(\mathbf{b},\mathbf{c}_{j,n}^{i})$, then $|\mu(\varphi_{i}(x,\mathbf{b}))-\frac{j}{n}|<\frac{1}{n}$,
\item if $|\mu(\varphi(x,b)) - \frac{k}{n}| < \frac{1}{n}$, then $\mathcal{M} \models \bigvee_{j=k-1}^{k+1} d_{j,n}^{\varphi_i}(\mathbf{b},\mathbf{c}_{j,n}^{i})$,
\item and if $l < 0$ or $l > n$, let $d_{l,n}^{\varphi_i}(y_i,\mathbf{c}_{l,n}^{i}) := y_i \neq y_i$. 
\end{enumerate} 
For simplicity, we will suppress the parameters, the $\mathbf{c}$'s, throughout the rest of the proof. Since $\mu$ is finitely approximated, it follows from Proposition \ref{fam:uniform} that for every $n \in \mathbb{N}$, there exists some $\overline{\mathbf{a}} \in (\mathcal{M})^{< \omega}$ such that for every $i \leq n$, 
\begin{equation*}  \sup_{\mathbf{b} \in \mathcal{M}^{|y_i|}} | \mu(\varphi_{i}(x,\mathbf{b})) - \Av(\overline{\mathbf{a}})(\varphi_{i}(x,\mathbf{b}))| < \frac{1}{n}. \tag{$\dagger$}
\end{equation*} 
We now construct a sentence similar to the one found in Proposition \ref{types}. We claim that for every $n \in \mathbb{N}$, there exists a number $r(n)$ such that the following sentence is true in $\mathcal{M}$. 
\begin{align*} \mathcal{M} \models &  \exists x_1,...,x_{r(n)}  \\&   \bigwedge_{i \leq n} \Biggl( \forall y_i  \biggl( \bigwedge_{j \leq n} \Big[d_{j,n}^{\varphi_i}(y_i) \to \bigvee_{A \in \Gamma(n,j)}  \Big( \bigwedge_{s \in A} \varphi_{i}(x_s,y_{i}) \wedge \bigwedge_{t \not \in A}  \neg \varphi_{i}(x_t,y_i)\Big) \Big] \\ & \wedge  \bigwedge_{A \subseteq \mathcal{P}([r(n)])} \Big[ \Big( \bigwedge_{l \in A} \varphi_i(x_l,y_i) \wedge \bigwedge_{l \not \in A} \neg \varphi_i(x_l,y_i) \Big) \to \bigvee_{t=-1}^{2}d^{\varphi_i}_{\gamma(A) + t,n}(y_i) \Big] \biggl) \Biggl), 
\end{align*} 
where $\Gamma(n,j) = \{ A \subseteq r(n): |\frac{|A|}{r(n)} - \frac{j}{n}| < \frac{2}{n}\}$ and $\gamma(A) = \floor{\frac{|A|}{r(n)} \cdot n} \cdot \frac{1}{n}$. This sentence essentially states that given some $n \in \mathbb{N}$, there exists points $\mathbf{a}_1,...,\mathbf{a}_{r(n)} \in \mathcal{M}$ such that if $i \leq n$ and $\models d_{j,n}^{\varphi_{i}}(\mathbf{b})$, then $\Av(\overline{\mathbf{a}})(\varphi_i(x,\mathbf{b})) \in (\frac{j-1}{n}, \frac{j+1}{n})$ and if $\Av(\overline{\mathbf{a}})(\varphi_i(x,b)) = r$, then the correct definition formula holds (modulo some wiggle room). 

We now justify why $\mathcal{M} \models \psi_n$. Let $\mathbf{a}_{1},...,\mathbf{a}_{r(n)}$ be points in $\mathcal{M}$ which satisfy equation $(\dagger)$. Now assume that $\mathcal{M} \models d_{j,n}^{\varphi_{i}}(\mathbf{b})$ for some $i \leq n$. Then $|\mu(\varphi_{i}(x,\mathbf{b})) - \frac{j}{n}|< \frac{1}{n}$. Since $\mathbf{a}_{1},...,\mathbf{a}_{r(n)}$ satisfy equation $(\dagger)$, we have that $|\mu(\varphi_{i}(x,\mathbf{b})) - \Av(\overline{\mathbf{a}})(\varphi_{i}(x,\mathbf{b}))| < \frac{1}{n}$ and hence $|\Av(\mathbf{a})(\varphi_i(x,\mathbf{b})) - \frac{j}{n}| < \frac{2}{n}$ and so 
\begin{equation*}
    \Big|\frac{|\{i \leq r(n): \models \varphi_i(\mathbf{a}_i,\mathbf{b})\}|}{r(n)} - \frac{j}{n}\Big| < \frac{2}{n}.
\end{equation*}
Now suppose that 
\begin{equation*}\bigwedge_{A \subseteq \mathcal{P}([r(n)])} \bigwedge_{l \in A} \varphi_{i}(\mathbf{a}_{l}, \mathbf{b}) \wedge \bigwedge_{t \not \in A} \neg \varphi_{i}(\mathbf{a}_{t},\mathbf{b}).
\end{equation*}
Then $\Av(\overline{\mathbf{a}})(\varphi(x,\mathbf{b})) = \frac{|A|}{r(n)}$. By equation $(\dagger)$, notice $|\mu(\varphi_{i}(x,\mathbf{b})) - \frac{|A|}{r(n)}| < \frac{1}{n}$, and so $| \mu(\varphi_{i}(x,\mathbf{b}))  - \frac{\frac{|A|}{r(n)} \cdot n}{n}| < \frac{1}{n}$. Hence by Fact \ref{fact:elementary},  

\begin{equation*} 
| \mu(\varphi_{i}(x,\mathbf{b}))  - \frac{ \floor{ \frac{|A|}{r(n)} \cdot n }}{n}| < \frac{1}{n} \text{ or } | \mu(\varphi_{i}(x,\mathbf{b}))  - \frac{ \lc \frac{A}{r(n)} \cdot n \rc }{n}| < \frac{1}{n}. 
\end{equation*} 
By construction, this implies that
\begin{equation*} 
\mathcal{M} \models \bigvee_{j= -1}^{1}d_{\gamma(A) + j,n}^{\varphi_i}(\mathbf{b}) \vee \bigvee_{j= -1}^{1}d_{\gamma(A) +1 + j,n}^{\varphi_i}(\mathbf{b}). 
\end{equation*}  
which is precisely what we want. 

As in the proof of Proposition \ref{types}, we construct the collection of sets $E_n$ where $E_n = \{k \in \mathbb{N}: M_{k} \models \psi_{n}\}$ with the parameters in $\psi_{n}$ replaced with the appropriate parameters in each $M_{k}$. By \strokeL o\'{s}'s theorem, $E_n \in D$.  Since $D$ is countably incomplete, we let $\{X_i : i \in \mathbb{N}\} \subseteq D$ such that $\bigcap_{i \in \mathbb{N}} X_i = \emptyset$. Again, we set $Y_n = E_n \cap X_n$ and we notice that $\bigcap Y_n = \emptyset$. 
In a similar way as to Proposition \ref{types}, we now construct a sequence of Keisler measure $(\mu_{i})_{i \in I}$ where each $\mu_{i}$ is in $\mathfrak{M}_{x}(M_i)$. We define the map $g: I \to \mathbb{N}$ where $g(i) = \max\{ n \in \mathbb{N}: i \in Y_n\}$ and $0$ otherwise. If $g(i) = 0$, we let $\mu_{i}$ be any measure in $\conv(M_i)$. If $g(i) = n > 0$, then we choose $\overline{a}_{i} = (a_1^{i},...,a_{r(n)}^{i})$ in $M_i$ such that $M_i \models \sigma_{n}(\overline{a}_{i})$ and let $\mu_{i} = \Av(\overline{a}_i)$. 

We claim that the ultralimit of $(\mu)_{i \in I}$ is equal to $\mu$. Let $\theta(x,y)$ be an $\mathcal{L}$-formula and assume that $\mu(\theta(x,\mathbf{e})) = s$. Fix $\epsilon > 0$ and let $[(e_i)]_D = \mathbf{e}$. It suffices to show that $K = \{i \in I: |\mu_{i}(\theta(x,e_{i})) - s| < \epsilon \} \in D$. Notice that $\theta(x,y)$ appears in our index, say $\theta(x,y) = \varphi_{m}(x,y_m)$. Let $m_{*} = \max\{ m +1, \lc \frac{10}{\epsilon} \rc \}$.
 Since $\mu(\theta(x,\mathbf{e})) = s$, notice that

 \begin{equation*} 
0 = |\mu(\varphi_{m}(x,\mathbf{e})) - s| = |\mu(\varphi_{m}(x,\mathbf{e})) - \frac{s \cdot m_*}{m_*}| < \frac{1}{m_*}
 \end{equation*} 
 and so by Fact \ref{fact:elementary}, 
 
 \begin{equation*} 
 |\mu(\varphi_{m}(x,\mathbf{e})) - \frac{\floor{s \cdot m_*}}{ m_*}| < \frac{1}{m_*} \text{ or }  |\mu(\varphi_{m}(x,\mathbf{e})) - \frac{\lc s \cdot m_* \rc}{ m_*}| < \frac{1}{m_*}. 
 \end{equation*} 
 
Let $j_* = \floor{s \cdot m_*}$. Then by construction $\mathcal{M} \models \bigvee_{\alpha = -1}^{2} d^{\varphi_{m}}_{j_* +\alpha,m_*}(\mathbf{e})$ and so for some $i_* \in \{j_* -1, j_*, j_* +1, j_* +2\}$,  $\mathcal{M} \models d_{i_*,m_*}^{\varphi_m}(\mathbf{e})$. Let $Z= \{i \in I: M_i \models d_{i_*,m_{*}}^{\varphi_{m}}(e_i)\}$. By \strokeL o\'s's theorem, $Z \in D$. We now argue that $Y_{m_*} \cap Z \subset K$. This which implies $K \in D$ and completes the proof. 

Let $k \in Z \cap Y_{m_*}$. Since $k \in Y_{m_{*}}$, it follows that  
\begin{equation*}
    M_{k} \models d_{i_{*},m_{*}}^{m}(e_{k}) \to \bigvee_{A \in \Gamma(m_{*},i_{*})} \Big( \bigwedge_{l \in A} \varphi_{m}(a_{l}^{k},e_{k}) \wedge \bigwedge_{t \not \in A} \neg \varphi_{m}(a_{t}^{k},e_{k}) \Big). 
\end{equation*}
Since $k \in Z$, it follows that for some $A_{*} \in \Gamma(m_{*},i_{*})$ 
\begin{equation*}
    M_{k} \models \bigwedge_{l \in A_{*}} \varphi_{m}(a_{l}^{k},e_{k}) \wedge \bigwedge_{t \not \in A_*} \neg \varphi_{m}(a_{t}^{k},e_{k}).
\end{equation*}
Now since $A_{*} \in \Gamma(m_{*},i_{*})$, it follows that $|\frac{|A_{*}|}{r(m_*)} - \frac{i_{*}}{m_{*}}| < \frac{2}{m_*}$. By definition, this implies that 
\begin{equation*}
    |\mu_{k}(\varphi_{m}(x,e_{k})) - \frac{i_{*}}{m_{*}}| < \frac{2}{m_{*}}. 
\end{equation*}
Now notice
\begin{equation*}
    |\mu_{k}(\varphi_{m}(x,e_k)) - s| \leq  |\mu_{k}(\varphi_{m}(x,e_k)) - \frac{i_*}{m_*}| + |\frac{i_*}{m_*} - s| < \frac{2}{m*} + \frac{4}{m_*} < \frac{10}{m_*} < \epsilon. 
\end{equation*}
Hence $k \in K$ and we conclude $K \in D$. 
\end{proof}

The following corollary follows tautologically from Theorems \ref{fact:ultra} and \ref{Theorem:pseudo}. 

\begin{corollary} Suppose that $\mathcal{L}$ is countable, $I = \kappa$, $(M_i)_{i \in I}$ is an indexed family of $\mathcal{L}$-structures, $D$ is a countably incomplete ultrafilter on $I$, $\mathcal{M} = \prod_{D} M_i$ is NIP, and $\mu \in \mathfrak{M}_{x}(\mathcal{M})$. Then $\mu$ is generically stable if and only if $\mu$ is pseudo-finite. 
\end{corollary}

We now discuss a brief theoretical application of the previous theorem. Let $T_n$ be the theory of the $n$-Henson graph \cite{Henson}. More explicitly, $T_n$ is the theory of the Fra\"{i}ss\'{e} limit of all finite $K_n$-free graphs. It is open whether or not this theory is pseudo-finite.

\begin{corollary} Suppose that $T_n$ is pseudo-finite. Let $(G_i)_{i \in \mathbb{N}}$ be a countable sequence of graphs and $D$ an ultrafilter on $\mathbb{N}$ such that $\mathcal{G} := \prod_{D} G_i \models T_n$. Consider the measure $\delta_{p}$ where $p$ is the unique type over  $\mathcal{G}$ such that $p \supset \{\neg R(x,\mathbf{b}): \mathbf{b} \in \mathcal{G}\}$. It was shown in \cite{CoGa} that $\delta_{p}$ is finitely approximated. Hence by the previous theorem, if $T_n$ is pseudo-finite then $\delta_p$ is itself pseudo-finite. 
\end{corollary} 

\section{The Morley product and the pseudo-finite product}

The main goal of this section is to show that the Morley product and the pseudo-finite product agree for reasonable pairs of measures. Explicitly, we prove that these products agree when the \textit{left-hand-side} measure is $\mathcal{M}$-definable and both are pseudo-finite. If both measures are definable and pseudo-finite, then the Morley product commutes on this pair. We then use these results to show that certain natural measures are not pseudo-finite over non-trivial ultraproducts of the Paley graphs. We begin this section by recalling the definitions of both the Morley and pseudo-finite products. 

\begin{definition}\label{def:morley} Let $\mu \in \mathfrak{M}_{x}(M)$, $\nu \in \mathfrak{M}_{y}(M)$, and suppose $\mu$ is definable. Then the \textbf{Morley product} of $\mu$ and $\nu$, denoted $\mu \otimes \nu$, is the unique measure in $\mathfrak{M}_{xy}(M)$ such that for any $\varphi(x,y,c) \in \mathcal{L}_{xy}(M)$,
\begin{equation*}
\mu \otimes \nu(\varphi(x,y,c)) = \int_{S_{y}(M)} F_{\mu,M}^{\varphi_{c}} d\tilde{\nu},
\end{equation*}
where $F_{\mu,M}^{\varphi_{c}}:S_{y}(M) \to [0,1]$ is the unique continuous map extending $\tp(b/M) \to \mu(\varphi(x,b,c))$ and $\tilde{\nu}$ is the unique regular Borel probability measure on $S_{y}(M)$ such that for any clopen set $[\psi(y)] \subseteq S_{y}(M)$, $\tilde{\nu}([\psi(y)]) = \nu(\psi(y))$. We will often drop the \textit{tilde} and write $\tilde{\nu}$ simply as $\nu$. 
\end{definition} 

We remind the reader that in general, the Morley product is non-commutative and can be a nuisance. We now recall/define the pseudo-finite product of Keisler measures. In this paper, we want our ultralimits of Keisler measures to be honest-to-goodness Keisler measures over the ultraproduct and so we do not care about internal subsets of our ultraproduct (only definable ones).

\begin{definition}\label{def:pf} Suppose that $(M_i)_{i \in I}$ is a sequence of $\mathcal{L}$-structures, $D$ is an ultrafiliter on $I$, and $\mathcal{M} := \prod_{D} M_i$. Suppose that $\mu \in \mathfrak{M}_{x}(\mathcal{M})$ and $\nu \in \mathfrak{M}_{y}(\mathcal{M})$ and both $\mu$ and $\nu$ are pseudo-finite. So there exists sequences of measures $(\mu_i)_{i \in I}$ and $(\nu_i)_{i \in I}$ such that $\mu_i \in \conv_{x}(M_i)$, $\nu_i \in \conv_{y}(M_i)$, $\lim_{D} \mu_i = \mu$, and $\lim_{D} \nu_i = \nu$. Then the \textbf{pseudo-finite product} of $\mu$ and $\nu$, denoted $\mu \boxtimes \nu$, is the unique measure in $\mathfrak{M}_{xy}(\mathcal{M})$ such that for any $\varphi(x,y,\mathbf{c}) \in \mathcal{L}_{xy}(\mathcal{M})$,
\begin{equation*}
\mu \boxtimes \nu (\varphi(x,y,\mathbf{c})) = \lim_{D} \Big(\mu_i \otimes \nu_i(\varphi(x,y,\mathbf{c}(i))\Big). 
\end{equation*}
\end{definition} 

\begin{remark}\label{product} In the definition above, since $\mu_i \in \conv_{x}(M_i)$ and $\nu_i \in \conv_{y}(M_i)$, we have that $\mu_i = \sum_{k=1}^{n} r_k \delta_{a_k}$ and $\nu_i = \sum_{j=1}^{m} s_j \delta_{b_j}$. One can check that for any $\theta(x,y) \in \mathcal{L}_{xy}(M_i)$, 
\begin{equation*} 
\mu_i \otimes \nu_i(\theta(x,y)) = \sum_{k=1}^{n} \sum_{j=1}^{m} r_ks_j \delta_{(a_k,b_j)}(\theta(x,y)) = \nu_i \otimes \mu_i (\theta(x,y)).  
\end{equation*} 
\end{remark}

One could be (and probably should be) worried about whether or not the pseudo-finite product is well-defined. We claim that it is well-defined and work towards proving this now. Towards this goal we start with a definition and prove a general lemma. 

\begin{definition}  Fix a sequence of $\mathcal{L}$-structures $(M_i)_{i \in I}$ and an ultrafilter $D$ on $I$. Let $\mathcal{M} = \prod_{D} M_i$. For each $i \in I$, let $g_i: S_{y}(M_i) \to [0,1]$ be a function. Then we define the map $g_{D}: \mathcal{M} \to [0,1]$ via $g_{D}(\mathbf{a}) = \lim_{D} (g_i(\tp(a_i/M_i)))$. 
\end{definition} 

It is easy to check that the above definition is well-defined. The following lemma is the main lemma of this section. We assume it is more or less folklore.  

\begin{lemma}\label{lemma:general} Fix a sequence of $\mathcal{L}$-structures $(M_i)_{i \in I}$ and an ultrafilter $D$ on $I$. Let $\mathcal{M} = \prod_{D} M_i$. For each $i \in I$, let $g_i,h_i : S_{y}(M_i) \to [0,1]$ be continuous functions and suppose that $\sup_{\mathbf{a} \in \mathcal{M}} |g_{D}(\mathbf{a}) - h_{D}(\mathbf{a})| < \epsilon$. If $(\nu_i)_{i \in I}$ is any sequence of Keisler measures, then 
\begin{equation*}
\lim_{D} \left(\int_{S_{y}(M_i)} g_i d\nu_i \right) \approxeq_{\epsilon} \lim_{D} \left(\int_{S_{y}(M_i)} h_i d\nu_i \right). 
\end{equation*} 
\end{lemma}

\begin{proof} Let $A = \{i \in I: \sup_{q \in S_{y}(M_i)} |g_i - h_i|(q) \leq \epsilon\}$. We claim $A \in D$. If not, then $I \backslash A \in D$. For each $i \in I \backslash A$, there exists some $q_i^* \in S_{y}(M_i)$ such that $|g_i - h_i|(q_i^{*}) > \epsilon$. Since the map $|g_i - h_i|:S_{y}(M_i) \to [0,1]$ is continuous and $M_i^{y}$ is a dense subset of $S_{y}(M_i)$, there exists some $e_i \in M_i^{y}$ such that $|g_i - h_i|(\tp(e_i/M_i)) > \epsilon$.  We define $(c_i)_{i \in I}$ where 
\begin{equation*} 
c_{i}=\begin{cases}
\begin{array}{cc}
e_{i} & i\in I\backslash A,\\
\text{anything in \ensuremath{M_{i}}} & \text{otherwise}.
\end{array}\end{cases}
\end{equation*} 
and let $\mathbf{c} = [(c_i)]_{D}$. Then by Fact \ref{computation}, 
\begin{equation*} 
\epsilon \leq \lim_{D} |g_i - h_i|(\tp(c_i/M_i)) = |\lim_D g_i(\tp(c_i/M_i)) - \lim_{D}h_i(\tp(c_i/M_i))|
\end{equation*} 
\begin{equation*} 
=| g_{D}(\mathbf{c}) - h_{D}(\mathbf{c})| < \sup_{\mathbf{a} \in \mathcal{M}} | g_{D}(\mathbf{a}) - h_{D}(\mathbf{a})| < \epsilon, 
\end{equation*}
which is a contradiction. Therefore we conclude that $A \in D$. 

Now consider the set $\tilde{A} = \{i \in I: |\int_{S_{y}(M_i)} g_i d\nu_i - \int_{S_{y}(M_i)} h_i d\nu_i|  \leq  \epsilon\}$. We claim that $\tilde{A} \supseteq A$ and so $\tilde{A} \in D$. Observe that if $i \in A$, then 

\begin{align*} 
\left|\int_{S_{y}(M_i)} g_i d\nu_i -\int_{S_{y}(M_i)} h_i d\nu_i \right| & \leq \int_{S_{y}(M_i)} |g_i - h_i| d\nu_i  \\ & \leq  \int_{S_{y}(M_i)} \epsilon d\nu_i \\ & =  \epsilon. 
\end{align*}
Thus $i \in \tilde{A}$ and we conclude $\tilde{A} \in D$. Now by $(5)$ of Fact \ref{computation}, 

\begin{equation*}\lim_{D} \left(\int_{S_{y}(M_i)} g_i d\nu_i \right) \approxeq_{\epsilon} \lim_{D} \left(\int_{S_{y}(M_i)} h_i d\nu_i \right). \qedhere
\end{equation*} 
\end{proof}

\begin{lemma}\label{product:2} Fix a sequence of $\mathcal{L}$-structure $(M_i)_{i \in I}$ and an ultrafilter $D$ on $I$. Let $\mathcal{M} = \prod_{D} M_i$. Let $\mu \in \mathfrak{M}_{x}(\mathcal{M})$, $\nu \in \mathfrak{M}_{y}(\mathcal{M})$ and suppose there exists $\mu^{1}_i,\mu^{2}_i \in \mathfrak{M}_{x}(M_i)$, $\nu^{1}_i,\nu^{2}_i \in \mathfrak{M}_{y}(M_i)$ such that 
\begin{enumerate}
    \item $\lim_{D} \mu^{1}_i = \lim_{D} \mu^{2}_i = \mu$ and $\lim_D \nu^{1}_i = \lim_{D} \nu^{2}_i = \nu$.
    \item For each $i \in I$, $\mu^{1}_i,\mu^{2}_i,\nu^{1}_i,\nu^{2}_i$ are each $M_i$-definable. 
    \item For each $i \in I$, $\mu^{2}_i \otimes \nu^{1}_i  = \nu^{1}_i \otimes \mu^{2}_i$ and $\mu^{2}_i \otimes \nu^{2}_i  = \nu^{2}_i \otimes \mu^{2}_i$
\end{enumerate}
Then $\lim_{D} ( \mu^{1}_i \otimes \nu^{1}_i )= \lim_{D} (\mu^{2}_i \otimes \nu^{2}_i)$. 
\end{lemma} 

\begin{proof} By Fact \ref{prop:well}, it suffices to show that for any $\varphi(x,y,\mathbf{b}) \in \mathcal{L}_{xy}(\mathcal{M})$,
\begin{equation*}
 \lim_{D} \mu_i^1 \otimes \nu_i^1 (\varphi(x,y,\mathbf{b}(i))) = \lim_{D} \mu_i^2 \otimes \nu_i^2 (\varphi(x,y,\mathbf{b}(i))) 
\end{equation*}
Since $\mu_i^{1},\mu_i^{2},\nu_i^{1},$ and $\nu_i^{2}$ are each $M_i$-definable, the maps $F_{\mu_{i}^{1}}^{\varphi_{\mathbf{b}(i)}}, F_{\mu_{i}^{2}}^{\varphi_{\mathbf{b}(i)}}: S_{y}(M_i) \to [0,1]$ and $F_{\nu_{i}^{1}}^{\varphi^{*}_{\mathbf{b}(i)}}, F_{\nu_{i}^{2}}^{\varphi^{*}_{\mathbf{b}(i)}}: S_{x}(M_i) \to [0,1]$ are each continuous. Moreover, notice that for any $\mathbf{a} \in \mathcal{M}$, 
\begin{equation*} \left(F_{\mu_i^{1},M_i}^{\varphi_{\mathbf{b}(i)}}\right)_{D}(\mathbf{a}) = \lim_{D} \mu_{i}^{1}(\varphi(x,\mathbf{a}(i),\mathbf{b}(i))) = \mu(\varphi(x,\mathbf{a},\mathbf{b}))
\end{equation*} 
\begin{equation*} = \lim_{D} \mu_{i}^{2}(\varphi(x,\mathbf{a}(i),\mathbf{b}(i))) =  \left(F_{\mu_i^{2},M_i}^{\varphi_{\mathbf{b}(i)}}\right)_{D}(\mathbf{a}). 
\end{equation*} 
Likewise, $ \left(F_{\nu_i^{1},M_i}^{\varphi^{*}_{\mathbf{b}(i)}}\right)_{D} = \left(F_{\nu_i^{2},M_i}^{\varphi^{*}_{\mathbf{b}(i)}}\right)_{D}$. Therefore, we can apply the Lemma \ref{lemma:general} (to two different pairs of sequences of functions) and with any $\epsilon > 0$. We now have the following sequence of equations: 
\begin{equation*}
\lim_{D} \mu^{1}_i \otimes \nu^{1}_i (\varphi(x,y,\mathbf{b}(i))) = \lim_{D} \left( \int_{S_{y}(M_i)}F_{\mu^{1}_i,M_i}^{\varphi_{\mathbf{b}(i)}} d\nu^{1}_i \right)
\end{equation*} 
\begin{equation*}
\overset{(a)}{=}  \lim_{D} \left( \int_{S_{y}(M_i)}F_{\mu^{2}_i,M_i}^{\varphi_{\mathbf{b}(i)}} d\nu^{1}_i \right) \overset{(b)}{=}  \lim_{D} \left( \int_{S_{y}(M_i)}F_{\nu^{1}_i,M_i}^{\varphi^{*}_{\mathbf{b}(i)}} d\mu^{2}_i \right)
 \end{equation*}
  \begin{equation*}
\overset{(a)}{=}  \lim_{D} \left( \int_{S_{y}(M_i)}F_{\nu^{2}_i,M_i}^{\varphi^{*}_{\mathbf{b}(i)}} d\mu^{2}_i \right) \overset{(b)}{=}  \lim_{D} \left( \int_{S_{y}(M_i)}F_{\mu^{2}_i,M_i}^{\varphi_{\mathbf{b}(i)}} d\nu^{2}_i \right)
 \end{equation*}
 \begin{equation*} 
 = \lim_{D}\mu_i^{2} \otimes \nu_i^{2}(\varphi(x,y,\mathbf{b}(i))). 
 \end{equation*} 
These equations are justified by the following:
 \begin{enumerate}[$(a)$]
 \item Lemma \ref{lemma:general} applied with arbitrarily small $\epsilon$. 
 \item This follows from condition (3) of our hypothesis. \qedhere
 \end{enumerate} 
\end{proof} 

\begin{proposition} The pseudo-finite product $\boxtimes$ is well-defined. 
\end{proposition}

\begin{proof} It suffices to argue that if $\mu \in \mathcal{M}_{x}(\mathcal{M})$, $\nu \in \mathfrak{M}_{y}(\mathcal{M})$, $\mu^{1}_i,\mu^{2}_i \in \conv_{x}(M_i)$, $\nu^{1}_i,\nu^{2}_i \in \conv_{y}(M_i)$, $\lim_{D} \mu^{1}_i = \lim_{D} \mu^{2}_i = \mu$, and  $\lim_{D} \nu_i^{1} = \lim_{D} \nu_i^{2} = \nu$ then $\lim_{D} (\mu_i \otimes \nu_i) = \lim_{D} (\mu_i \otimes \nu_i)$. Notice that every measure which concentrates on finitely many realized types is definable. By Remark \ref{product}, all necessary products commute. Hence we can apply Lemma \ref{product:2} and the proof is complete. 
\end{proof}

We now connect the Morley product and the pseudo-finite product and derive the main result of this section.

\begin{theorem}\label{MT2} Let $(M_i)_{i \in I}$ be a indexed family of $\mathcal{L}$-structures, $D$ be an ultrafilter on $I$, and $\mathcal{M} := \prod_{D} M_i$. Let $\mu \in \mathfrak{M}_{x}(\mathcal{M})$ and $\nu \in \mathfrak{M}_{y}(\mathcal{M})$. Suppose that both $\mu$ and $\nu$ are pseudo-finite. If $\mu$ is definable, then $\mu \otimes \nu = \mu \boxtimes \nu$.
\end{theorem} 

\begin{proof} Fix a formula $\varphi(x,y,\mathbf{d}) \in \mathcal{L}_{xy}(\mathcal{M})$, $\epsilon > 0$, and let $[(d_i)]_{D} = \mathbf{d}$. Choose $(\mu_i)_{i \in I}$ and $(\nu_i)_{i \in I}$ such that $\mu_i \in \conv_{x}(M_i)$, $\nu_i \in \conv_{y}(M_i)$,  $\lim_{D} \mu_i = \mu$, and $\lim_{D} \nu_i = \nu$.  Since $\mu$ is definable, there exists real numbers $r_1,...,r_n$ and $\mathcal{L}_{y}(\mathcal{M})$-formulas $\psi_{1}(y,\mathbf{a}_1),...,\psi_{n}(y,\mathbf{a}_n)$ which partition $\mathcal{M}^{y}$ such that 
\begin{equation*} 
\sup_{q \in S_{y}(\mathcal{M})}| F_{\mu,\mathcal{M}}^{\varphi}(q) - \sum_{j=1}^{n} r_j \mathbf{1}_{\psi_{j}(y,\mathbf{a}_j)}(q)| < \epsilon. \eqno{(\dagger)}
\end{equation*} 
Let $f_i := \sum_{j=1}^{n} r_j \mathbf{1}_{\psi_j(y,\mathbf{a}_j(i))}$ and so $f_i : S_{y}(M_i) \to [0,1]$.  Then we have that 
\begin{equation*} \mu \otimes \nu(\varphi(x,y,\mathbf{d})) = \int_{S_{y}(\mathcal{M})} F_{\mu,\mathcal{M}}^{\varphi_{\mathbf{d}}}d\nu \approx_{\epsilon} \int_{S_{y}(\mathcal{M})} \sum_{j=1}^{n} r_j \mathbf{1}_{\psi_j(y,\mathbf{a}_{j})} d\nu = \sum_{i=1}^{n} r_j \nu(\psi_j(y,\mathbf{a}_{j}))
\end{equation*}
\begin{equation*}
=  \sum_{j=1}^{n} r_j \left( \lim_{D} \nu_{i}(\psi_j(y,\mathbf{a}_{j}(i))) \right) = \lim_{D} \left( \sum_{j=1}^{n} r_j \nu_{i}(\psi_{j}(y,\mathbf{a}_j(i)) \right)
\end{equation*} 
\begin{equation*}
 = \lim_{D} \left( \int_{S_{y}(M_{i})} \sum_{j=1}^{n} r_j \mathbf{1}_{\psi_j(y,\mathbf{a}_j(i))} d\nu_i \right) = \lim_{D} \left( \int_{S_{y}(M_i)} f_i d\nu_i \right),
\end{equation*} 
\begin{equation*} 
\overset{(*)}{\approxeq_{\epsilon}} \lim_{D} \left( \int_{S_{y}(M_i)}F_{\mu_i,M_i}^{\varphi_{\mathbf{d}(i)}} d\nu_i\right) = \mu \boxtimes \nu (\varphi(x,y,\mathbf{d})). 
\end{equation*} 
We justify equation $(*)$ using Lemma \ref{lemma:general}. Clearly for each $i \in I$, we have that $F_{\mu_i,M_i}^{\varphi_{\mathbf{d}(i)}}$ and $f_i$ are continuous. Moreover, for any $\mathbf{c} \in \mathcal{M}$, 
\begin{equation*} 
|(F_{\mu_i,M_i}^{\varphi_{\mathbf{d}(i)}})_{D}(\mathbf{c}) - f_{D}(\mathbf{c})| = |\mu(\varphi(x,\mathbf{c},\mathbf{d})) - \sum_{j=1}^{n} r_j \mathbf{1}_{\psi_{j}(y,\mathbf{a}_j)}(\tp(\mathbf{c}/\mathcal{M}))| < \epsilon. 
\end{equation*} 
The first equation above is a direct consequence of the definitions (left term) and \strokeL o\'s's theorem (right term). Hence Lemma \ref{lemma:general} applies and approximation $(*)$ is justified. 
\end{proof} 

\begin{corollary}\label{fun}Let $(M_i)_{i \in I}$ be a indexed family of $\mathcal{L}$-structures, $D$ be an ultrafilter on $I$, and $\mathcal{M} := \prod_{D} M_i$. Let $\mu \in \mathfrak{M}_{x}(\mathcal{M})$ and $\nu \in \mathfrak{M}_{y}(\mathcal{M})$. Suppose that both $\mu$ and $\nu$ are pseudo-finite. If $\mu$ and $\nu$ are both definable, then $\mu \otimes \nu = \nu \otimes \mu$.
\end{corollary} 
\begin{proof} Fix a formula $\varphi(x,y,\mathbf{d}) \in \mathcal{L}_{xy}(\mathcal{M})$. Let $\mu_i \in \conv_{x}(M_i)$ and $\nu_i \in \conv_{y}(M_i)$ such that $\lim_{D} \mu_i = \mu$ and $\lim_{D} \nu_i = \nu$. Then Theorem \ref{MT2}  and Remark \ref{product} yield the following, 
\begin{equation*} \mu \otimes \nu(\varphi(x,y,\mathbf{d})) = \mu \boxtimes \nu(\varphi(x,y,\mathbf{d})) = \lim_{D} \mu_i \otimes \nu_i(\varphi(x,y,\mathbf{d}(i))) 
\end{equation*} 
\begin{equation*} = \lim_{D} \nu_i \otimes \mu_i(\varphi(x,y,\mathbf{d}(i))) = \nu \boxtimes \mu (\varphi(x,y,\mathbf{d})) = \nu \otimes \mu(\varphi(x,y,\mathbf{d})). \qedhere
\end{equation*} 
\end{proof} 

\begin{remark}\label{remark:anti} The previous results goes through \textit{locally}. In particular, suppose that $\varphi(x,y)$ is an $\mathcal{L}$-formula, $(M_i)_{i \in I}$ is an indexed family of $\mathcal{L}$-structures, $D$ is an ultrafilter on $I$, and $\prod_{D} M_i = \mathcal{M}$. Suppose that $\mu$ and $\nu$ are pseudo-finite and $\mu$ is $\varphi$-definable; i.e. we only assume that the map $F_{\mu,\mathcal{M}}^{\varphi}: S_{y}(\mathcal{M}) \to [0,1]$ is well-defined and continuous. 
\begin{enumerate} 
\item Then $\mu \otimes \nu(\varphi(x,y)) = \mu \boxtimes \nu(\varphi(x,y))$. 
\item If $\nu$ is $\varphi^{*}$-definable, then $\mu \otimes \nu(\varphi(x,y)) = \nu \otimes \mu(\varphi(x,y))$. 
\end{enumerate} 
\end{remark}

\subsection{An application} Remark \ref{remark:anti} can be used to show that many measures are \textit{not} pseudofinite (with respect to a particular sequence of $\mathcal{L}$-structures and choice of ultrafilter). Here we consider the Paley graphs and demonstrate that the ``unfair coin-flipping measures" are not pseudo-finite in this context. 

\begin{definition}[Paley Graphs] Let $(q_i)_{i \in \mathbb{N}}$ be the sequence of primes such that $q_i = 1$ mod $4$. Let $\mathbb{F}_{q_i}$ be the finite field of size $q_i$ and  $G_{q_i} = (\mathbb{F}_{q_i}; R)$ where $G_{q_i} \models R(a,b)$ if and only if there exists some $z \in \mathbb{F}_{q_i}$ such that $z^{2} = a -b$. We remark that $G_{q_i}$ is a graph. 
\end{definition}  

\begin{fact}\label{fact:2} The following statements are true. 
\begin{enumerate} 
\item For any non-principal ultrafilter $D$ on $I$, $\prod_{D} G_{q_i}$ is a model of the random graph \cite[Example 3.4]{Mac}.
\item For each $q_i$ and $a \in G_{q_i}$, the degree of $a$ is $\frac{q_i - 1}{2}$ \cite{wiki}.
\end{enumerate} 
\end{fact} 

\begin{definition} Let $M$ be a model of the random graph. We say that a measure $\nu$ is an unfairly weighted coin-flipping measure if $\nu = \nu_{p}$ where $p \in [0,1] \backslash \{1/2\}$ and for any finite sequence of distinct singletons $a_1,...,a_n,b_1,...,b_m$, 
\begin{equation*}
    \nu_{p}\left( \bigwedge_{i=1}^{n} R(x,a_i) \wedge \bigwedge_{j=1}^{m} \neg R(x,b_j) \right) = \left(\frac{1}{p}\right)^{n} \left(\frac{1}{1-p}\right)^{m}. 
\end{equation*}
By quantifier elimination, $\nu_{p}$ is unique.
\end{definition}

\begin{theorem}\label{Paley} Suppose that $\mathcal{L} = \{R\}$, $I = \mathbb{N}$, $M_i = G_{q_i}$, $D$ is a non-principal ultrafilter on $I$, and $\mathcal{M} : = \prod_{D} M_i$. We let $\mu$ be the pseudo-finite counting measure, i.e. $\mu = \lim_{D} \mu_i$ where for each $\varphi(x) \in \mathcal{L}_{x}(M_i)$, $\mu_i(\varphi(x)) = \frac{|\{a \in M_i: M_i \models \varphi(a)\}| }{|M_i|}$. 
\begin{enumerate}
\item The measure $\mu$ is $R$-definable and in particular $F_{\mu,\mathcal{M}}^{R}$ is the constant function which outputs $1/2$.
\item For any measure $\nu$ such that $F_{\nu,\mathcal{M}}^{R}$ is a constant function not equal to 1/2, $\nu$ is not pseudo-finite (with respect to this indexing family and ultrafilter). In particular, any unfairly weighted coin-flipping measures is not pseudo-finite in this context.
\end{enumerate}
\end{theorem} 

\begin{proof} We prove the statements. 
\begin{enumerate}
\item By (3) of Fact \ref{fact:2}, $\mu(R(x,\mathbf{a})) = \frac{1}{2}$ for each $\mathbf{a} \in \mathcal{M}$. Obviously this map can be extended to the constant function $1/2$ over $S_{y}(\mathcal{M})$ which is continuous. 
\item Suppose $\nu$ is pseudo-finite. Assume that for every $\mathbf{a} \in \mathcal{M}$, $\nu(R(x,\mathbf{a})) = r \neq 1/2$. Then $\nu$ is $R^{*}$-definable and 
\begin{equation*} 
\nu \otimes \mu (R(x,y)) = r \neq 1/2 = \mu \otimes \nu(R(x,y)). 
\end{equation*} 
This contradicts Remark \ref{remark:anti}. \qedhere
\end{enumerate}
\end{proof} 

\section{Constructing idempotent measures on monster models of pseudo-finite groups}
We now wish to work in the ``standard'' context for Keisler measures, i.e. measures over a monster model. In this section, we fix a monster model $\mathcal{U}$ in the background and ultraproducts will always be small models of $\mathcal{U}$. The purpose of this section is to construct idempotent generically stable measures over pseudo-finite NIP groups. Idempotent Keisler measures play an interesting role in the model theory of groups and have been previously studied in \cite{ChernGan,ChernGan2}. 

\begin{definition} Suppose that $\mathcal{U}$ is a monster model of a group and $G$ is a small submodel. Suppose that $\mu \in \mathfrak{M}_{x}(\mathcal{U})$, $\nu \in \mathfrak{M}_{y}(\mathcal{U})$. We assume that $|x| = 1$.  
\begin{enumerate}
    \item We say that $\mu$ is definable over $G$ if $\mu|_{G}$ is $G$-definable and for any $\varphi(x,y) \in \mathcal{L}$ and $b \in \mathcal{U}$, $\mu(\varphi(x,b)) = F_{\mu|_{G},G}^{\varphi}(\tp(b/G))$. 
    \item If $\mu$ is definable (over $G$), we let $\mu \tilde{\otimes} \nu$ denote the standard Morley product (for global measures). In particular, for any formula $\varphi(x,y,c) \in \mathcal{L}_{xy}(\mathcal{U})$, we have that 
    \begin{equation*}
    \mu \tilde{\otimes} \nu (\varphi(x,y,b)) = \int_{S_{y}(Gb)} F_{\mu,Gb}^{\varphi_{b}} d(\nu|_{Gb}). 
    \end{equation*} 
    \item If $\mu$ is definable (over $G$) then $\mu * \nu$ is the unique measure in $\mathfrak{M}_{x}(\mathcal{U})$ such that for any formula $\varphi(x) \in \mathcal{L}_{x}(\mathcal{U})$, 
    \begin{equation*}
        \mu * \nu(\varphi(x)) = \mu \tilde{\otimes} \nu(\varphi(x \cdot y)). 
    \end{equation*}
    \item If $\mu$ is definable (over $G$) we say that $\mu$ is idempotent if $\mu * \mu = \mu$.  
\end{enumerate}
\end{definition} 

A proof of the following fact can be found in \cite[Theorem A.24]{CGH}. 

\begin{fact} Suppose that $\mu \in \mathfrak{M}_{x}(M)$ and $\mu$ is $M$-definable. Then there exists a unique measure $\hat{\mu} \in \mathfrak{M}_{x}(\mathcal{U})$ such that $\hat{\mu}|_M = \mu$ and $\hat{\mu}$ is definable over $M$.
\end{fact}

Finally, we recall the classification of idempotent measures over finite groups. This follows immediately from the main results in \cite{KI} and/or \cite{Wendel} since every finite group can be viewed as a compact group under the discrete topology.

\begin{fact}\label{fact:groups} Let $G$ be any finite group and $\mu$ be a measure on $G$. Then the following are equivalent: 
\begin{enumerate} 
\item $\mu \star \mu = \mu$, where $\star$ is the standard convolution product. 
\item There exists a subgroup $H$ of $G$ such that $\mu = \frac{1}{|H|}\sum_{h \in H} \delta_{h}$. 
\end{enumerate} 
\end{fact} 

\begin{theorem}\label{final} Let $(G_i)_{i \in I}$ be a sequence of groups and $H_i$ be any sequence of finite subgroups where $H_i \leq G_i$. Let $\mathcal{G} = \prod_{D} G_i$, $\mathcal{G} \subset \mathcal{U}$, $\mu_{i} := \frac{1}{|H_i|} \sum_{h \in H_i} \delta_{h}$, and $\mu = \lim_{D} \mu_{i}$. If $\mu$ is $\mathcal{G}$-definable, then $\hat{\mu}$ is idempotent. 
\end{theorem} 

\begin{proof} By definition, $\hat{\mu}$ is the unique definable measure which extends $\mu$ from $\mathcal{L}_{x}(\mathcal{G})$ to $\mathcal{L}_{x}(\mathcal{U})$. By \cite[Proposition 3.15]{ChernGan}, it follows that $\hat{\mu} * \hat{\mu}$ is definable over $\mathcal{G}$. Hence, if $(\hat{\mu} * \hat{\mu})|_{\mathcal{G}} = \hat{\mu}|_{\mathcal{G}}$, then $\hat{\mu} * \hat{\mu} = \hat{\mu}$. Let $\varphi(x) \in \mathcal{L}_{x}(\mathcal{G})$, $\varphi'(x,y) := \varphi(x \cdot y)$, and consider the following computation:
\begin{equation*} 
\hat{\mu} * \hat{\mu}(\varphi(x)) = \hat{\mu}_{x} \tilde{\otimes} \hat{\mu}_{y}(\varphi(x \cdot y)) = \int_{S_{y}(\mathcal{G})} F_{\hat{\mu},\mathcal{G}}^{\varphi'} d(\hat{\mu}|_{\mathcal{G}}) = \int_{S_{y}(\mathcal{G})} F_{\mu,\mathcal{G}}^{\varphi'} d\mu
\end{equation*} 
\begin{equation*} 
=\mu_x \otimes \mu_y (\varphi(x \cdot y)) = \mu_x \boxtimes \mu_y(\varphi(x \cdot y)) = \lim_{D} (\mu_{i} \otimes \mu_{i}(\varphi(x \cdot y)))
\end{equation*} 
\begin{equation*}
= \lim_{D} (\mu_{i} \star \mu_{i}(\varphi(x))) = \lim_{D} ( \mu_{i}(\varphi(x))) = \mu(\varphi(x)) = \hat{\mu}(\varphi(x)). \qedhere
\end{equation*} 
\end{proof}  

\begin{theorem}\label{gsim} Let $(G_i)_{i \in I}$ be a sequence of groups and $H_i$ be any sequence of finite subgroups where $H_i \leq G_i$. Let $\mathcal{G} = \prod_{D} G_i$, $\mathcal{G} \subset \mathcal{U}$, $\mu_{i} := \frac{1}{|H_i|} \sum_{h \in H} \delta_{h}$, and $\mu = \lim_{D} \mu_{i}$. If $\mathcal{G}$ is NIP, then $\hat{\mu}$ is a generically stable idempotent Keisler measure. 
\end{theorem}

\begin{proof}
Follows directly from Theorem \ref{fact:ultra} and Theorem \ref{final}.  
\end{proof}

\begin{question} Does the converse of Theorem \ref{gsim} hold? If $\mathcal{G} = \prod_{D} G_i$ is NIP and $\mu \in \mathfrak{M}_{x}(\mathcal{G})$ is a generically stable idempotent, does there exist a sequence of finite subgrups $H_i$ of $G_i$ such that $\lim_{D} \mu_i = \mu$ where $\mu_i = \frac{1}{|H_{i}|}\sum_{h \in H_{i}} \delta_{h}$. We remark that Theorem \ref{Theorem:pseudo} implies that $\mu$ is pseudo-finite, but does not directly imply that we can find a sequence of measures of the correct form. 

Assume furthermore that the that the normal subgroups of $(G_i)_{i \in I}$ are uniformly definable, i.e. there exists a single formula $\gamma(x,\bar{y})$ such that for any $i \in I$ and normal subgroup $H $ of $G_i$, there exists a parameter $b \in G_i^{y}$ such that $G_i \models \gamma(a,b)$ if and only if $a \in H$. Then one can show that $K: = \{i \in I: \lim_{n \to \infty} \mu_i^{(*n)}$ converges$\}$ is in $D$ (by using the equivalence of $(i)$ and $(v)$ in \cite[Corollary 2.1]{Hognas}). It is routine to check that if $\lim_{n \to \infty} \mu_i^{(*n)}$ converges, then it converges to an idempotent. Consider the sequence of measures 
\begin{equation*}
\nu_{i}:=\begin{cases}
\begin{array}{cc}
\lim_{n\to\infty}\mu_{i}^{(*n)} & i\in K,\\
\delta_{e} & otherwise.
\end{array}\end{cases}
\end{equation*}
Does $\lim_{D} \nu_i = \mu$? If not, under which conditions does the equalty hold? 
\end{question}

\end{document}